\documentclass[12pt]{article}

\usepackage{amsthm,amsmath,amssymb}

\usepackage{graphicx}

\usepackage[colorlinks=true,citecolor=black,linkcolor=black,urlcolor=blue]{hyperref}
\usepackage {tikz}
\usetikzlibrary{automata,arrows,calc,positioning}
\usetikzlibrary{ decorations.pathreplacing}

\newcommand{\real}{\mathbb{R}}
\newcommand{\Complex}{\mathbb{C}}

\theoremstyle{plain}
\newtheorem{theorem}{Theorem}[section]
\newtheorem{lem}{Lemma}[section]
\newtheorem{cor}{Corollary}[section]
\newtheorem{prop}{Proposition}[section]

\theoremstyle{definition}

\newtheorem{problem}{Problem}

\theoremstyle{remark}

\title{\bf On the Stability of \\
Independence Polynomials}

\author{Jason Brown\thanks{Supported by NSERC grant.}~ and Ben Cameron\\
Department of Mathematics \& Statistics\\
Dalhousie University\, Halifax, Canada}

\date{}

\begin{document}

\maketitle


\begin{abstract}
 The independence polynomial of a graph is the generating polynomial for the number of independent sets of each size, and its roots are called {\em independence roots}. We investigate the stability of such polynomials, that is, conditions under which the roots lie in the left half-plane (all of the real roots of independence polynomial are negative and hence lie in this half-plane). We show stability for all independence polynomials of graphs with independence number at most three, but for larger independence number we show that the independence polynomials can have roots arbitrarily far to the right. We provide families of graphs whose independence polynomials are stable and ones that are not, utilizing various graph operations. 

  \bigskip\noindent \textbf{Keywords:} graph; independent set; independence polynomial; stable polynomial; root
\end{abstract}

\section{Introduction}\label{Secintro}

A subset of vertices of a (finite, undirected and simple) graph $G$ is called {\em independent} if it induces a subgraph with no edges; the \emph{independence number} of  $G$ is the size of the largest independent set in $G$ and is denoted by $\alpha(G)$ (or just $\alpha$ if the graph is clear from context). The {\em independence polynomial} of $G$, denoted by $i(G,x)$, is defined by 
\[ i(G,x)=\sum_{k=0}^{\alpha}i_kx^k,\] 
where $i_k$ is the number of independent sets of size $k$ in $G$. We call the roots of $i(G,x)$ the {\em independence roots} of $G$. 

Research on the independence polynomial and in particular, the independence roots, has been very active (see, for example, \cite{brownwc,BrownHickman2002,INDROOTS,Brown2005,Csivari2013,Levit2008} and \cite{INDPOLY} for an excellent survey) since it was first defined by Gutman and Harary in 1983 \cite{INDFIRST} (including recent connections, in the multivariate case, to the hard core model in statistical physics \cite{ScottSokal2005}). On the nature of these roots, Chudnovsky and Seymour \cite{Chudnovsky2007} showed that the independence roots of claw-free graphs are all real, and Brown and Nowakowski \cite{Brown2005} showed that with probability tending to $1$, a graph will have a nonreal independence root.  

Asking when the independence roots are all real is a very natural question, but what about their location in the complex plane? While Brown et al. \cite{INDROOTS} showed that the collection of the independence roots of all graphs are in fact dense in the complex plane, plots of the independence roots of small graphs show a different story (see Figures~\ref{smallgraphsplot} and \ref{smalltreesplot}).  One striking thing about these plots is that not a single root lies in the open right half-plane (RHP) $\{z\in\Complex:\text{Re}(z) > 0\}$, so we are left to wonder: how ubiquitous are graphs with {\em stable} independence polynomials, that is, with all their independence roots in the left half-plane (LHP) $\{z\in\Complex:\text{Re}(z)\le0\}$? (A polynomial with all of its roots in the LHP is called {\em Hurwitz quasi-stable}, or simply {\em stable}, and such polynomials are important in many applied settings \cite{Choe2004}). Such a region is a natural extension of the negative real axis, which plays such a  dominant role in the Chudnovsky-Seymour result on claw-free graphs.

We shall call a graph itself {\em stable} if its independence polynomial is stable. It is known that the independence root of smallest modulus is always real and therefore negative (see \cite{brownwc}), so no independence polynomial has all its roots in the RHP, but it is certainly possible for it to have all roots in the LHP. This paper shall consider the stability of independence polynomials, providing some families of graphs whose independence polynomials are indeed stable, while showing that graphs formed under various constructions have independence polynomials that are not only nonstable but have roots with arbitrarily large real part. 

\begin{figure}[ht]
\begin{center}
\includegraphics[scale=0.5]{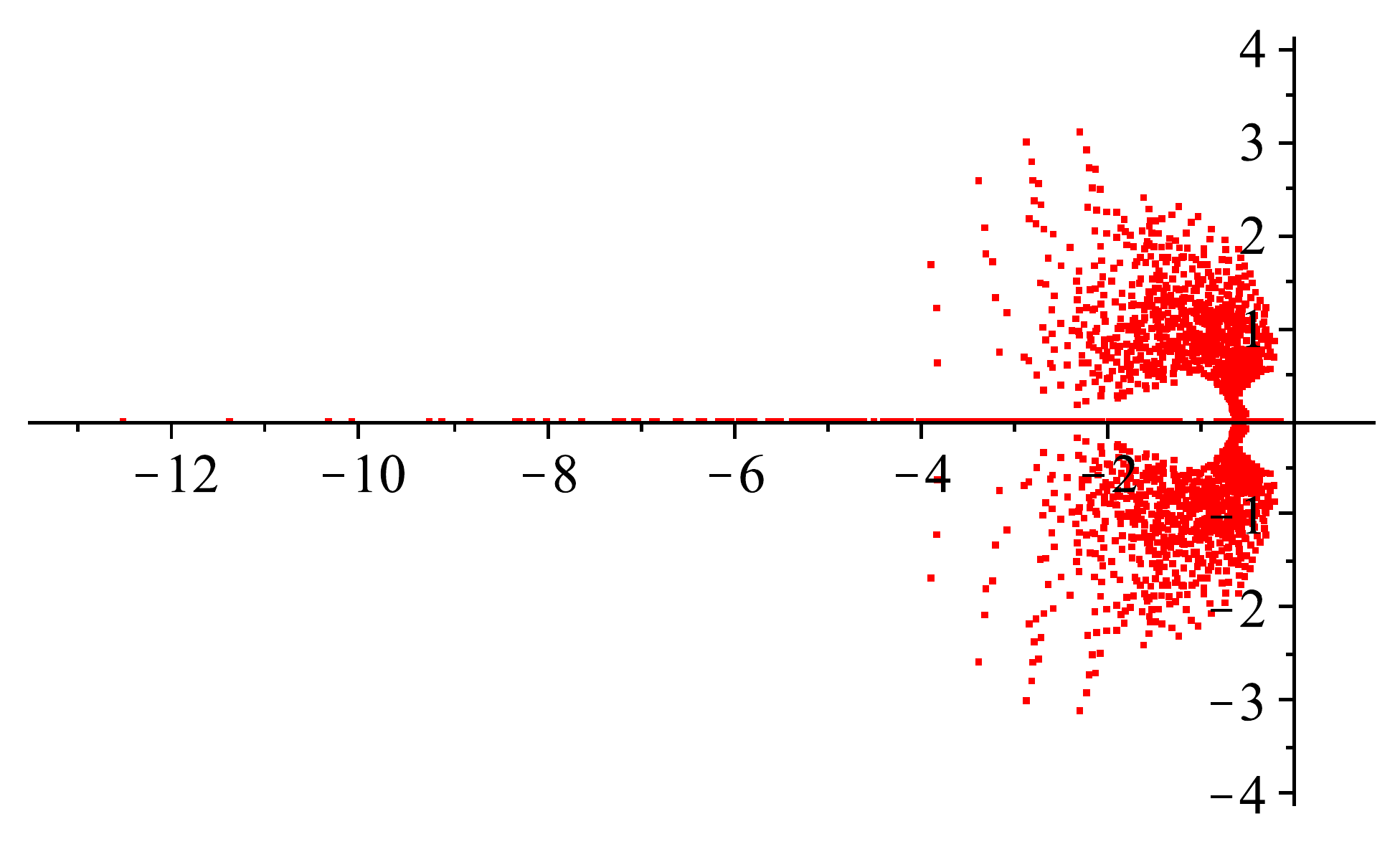}
\end{center}
\caption{Independence roots of all graphs on $9$ or fewer.}\label{smallgraphsplot}
\end{figure}

\begin{figure}[ht]
\begin{center}
\includegraphics[scale=.5]{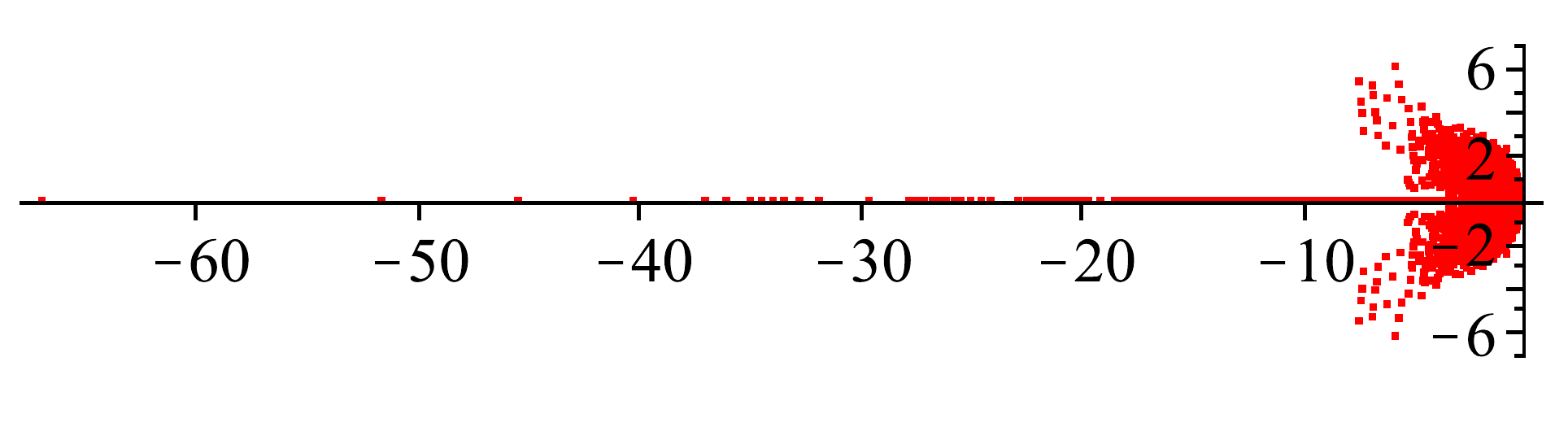}
\end{center}
\caption{Independence roots of all trees on $14$ or fewer vertices.}\label{smalltreesplot}
\end{figure} 

\vspace{0.2in}
We shall first consider stability for graphs with small independence number, and show that while all graphs with independence number at most $3$ are stable, it is not the case for larger independence number. Then we shall turn to producing stable graphs as well as nonstable graphs. Graph operations will play roles in both. We conclude with a few open questions.

\section{Stability for Small Independence Number}\label{secSmallIndNum}

We begin by proving that all graphs with independence number at most three are indeed stable. To do so, we shall utilize a necessary and sufficient condition, due to Hermite and Biehler, for a real polynomial to be stable. Prior to introducing the theorem, we shall need some notation.

A polynomial is {\it standard} when either it is identically zero or its leading coefficient is positive. Given a polynomial $P(x) = \displaystyle{\sum_{i=0}^{d}a_{i}x^{i}}$, let 
\[ P^{even}(x) = \sum_{i=0}^{\lfloor d/2 \rfloor} a_{2i} x^{i},\]
and
\[ P^{odd}(x) = \sum_{i=0}^{\lfloor (d-1)/2 \rfloor} a_{2i+1} x^{i};\]
$P^{even}(x)$ and $P^{odd}(x)$ are the ``even'' and ``odd'' parts of the polynomial, with
\[ P(x) = P^{even}(x^2) + xP^{odd}(x^2).\]
For example, if $P(x)=i(K_{3,3},x)=1+6x+6x^2+2x^3$, then $P^{even}(x)=1+6x$ and $P^{odd}(x)=6+2x$. 

Finally, let $f(x)$ and $g(x)$ be two real polynomials with all real roots, with say $s_1\le s_2\le\ldots\le s_n$ and $t_1\le t_2\le\ldots\le t_m$ being their respective roots. We say that 
\begin{itemize}
\item $f$ \textit{interlaces} $g$ if $m=n+1$ and $t_1\le s_1\le t_2\le s_2\le \cdots\le s_n\le t_{n+1}$, and 
\item $f$ \textit{alternates left} of $g$ if $m=n$ and $ s_1\le t_1\le s_2\le t_2\le \cdots\le s_n\le t_{n}$. 
\end{itemize}
We write $f\prec g$ for either $f$ interlaces $g$ or $f$ alternates left of $g$. A key result that we shall rely upon is the Hermite-Biehler Theorem which characterizes when a real polynomial is stable (see, for example, \cite{Wagner2000}).

\begin{theorem}[\bfseries Hermite-Biehler]\label{thmhb}
Let $P(x)=P^{even}(x^2)+xP^{odd}(x^2)$ be standard. Then $P(x)$ is stable if and only if both $P^{even}$ and $P^{odd}$ are standard, have only nonpositive roots and $P^{odd} \prec P^{even}$. \hfill $\Box$
\end{theorem} 

We are now in a position to prove:

\begin{prop}
If $G$ is a graph of order $n$ with $\alpha(G)\le 3$, then $i(G,x)$ is stable.
\end{prop}
\begin{proof}
For graphs with independence number $1$ (that is, a complete graph), the independence polynomial is of the form $1+nx$. These polynomials are obviously stable for all $n$.
For graphs with independence number $2$, the independence polynomial has the form $1+nx+i_2x^2$. The complement of a graph with independence number $2$ is triangle-free, and hence by Turan's famous theorem, has at most $\lfloor \frac{n}{2} \rfloor \lceil \frac{n}{2} \rceil \leq \frac{n^2}{4}$ many edges. However, clearly the number of edges in the complement is precisely $i_{2}$, so that $i_{2} \leq \frac{n^2}{4}$, which implies that the discriminant of  the independence polynomial $1+nx+i_2x^2$ is nonnegative, and the roots are real (and hence negative). Therefore, the independence polynomial of a graph with independence number $2$ is necessarily stable.

For graphs with independence number $3$, it is again the case that all independence polynomials are stable. To show this, we utilize the Hermite-Biehler Theorem. If $\alpha(G)=~3$, then $$i(G,x)=1+nx+i_2x^2+i_3x^3=P^{even}(x^2)+xP^{odd}(x^2)$$ where $P^{even}=1+i_2x$ and $P^{odd}=n+i_3x$. It is clear that $P^{even}$ and $P^{odd}$ each have only one real root, but we must show that $P^{odd}\prec P^{even}$, i.e. that $\tfrac{-n}{i_3}\le \tfrac{-1}{i_2}$. Equivalently, we need to show that  $ni_2\ge i_3$, but this follows as every independent set of size $3$ contains an independent set of size $2$, so adjoining an outside vertex to each independent set of size $2$ will certainly cover all independent sets of size $3$ at least once.
Thus by Theorem \ref{thmhb}, $i(G,x)$ is stable for all $\alpha(G)=3$.
\end{proof}

We now turn to independence number at least $4$, and show, in contrast, that there are many graphs whose independence roots lie in the RHP -- in fact, we can find roots in the RHP with arbitrarily large real part. We begin with a lemma. This lemma will be pivotal for many of the results in the remainder of this section as well as in Section \ref{secRHPgraphs}.

\begin{lem}\label{lemarblargeRHProots}
Let $R > 0 $ and $f(x)\in\real [x]$ be a polynomial of degree $d$ with positive coefficients. Then 
\begin{enumerate}
\item if $d \ge 4$, then for $m$ sufficiently large $f(x)+mx$ has a root with real part greater than $R$, and
\item if $d \ge 3$, then for $\ell$ sufficiently large $f(x)+\ell$ has a root with real part greater than $R$.
\end{enumerate}
\end{lem}
\begin{proof}
We consider the polynomial $g(x)=f(x+R)$. As per the Hermite-Biehler theorem let $g^{even}(x)$ and $g^{odd}(x)$ denote the even and odd part of $g(x)$, respectively, so that \[ g(x)=g^{even}(x^2)+xg^{odd}(x^2).\] 
For the proof of part 1, consider the polynomial $P_{m}(x)=m(x+R)+g(x)$. 
Clearly 
\[ P^{even}_{m}(x)=mR+g^{even}(x)\] 
and 
\[ P_m^{odd}(x)=m+g^{odd}(x).\] 
Suppose first that $d$ is even. As $d \geq 4$, clearly $\deg(g^{even}(x))\ge 2$. 
The leading coefficient of $P^{even}_{m}(x)$ is positive (as $f$ has all positive coefficients and $R>0$), so it follows that $\lim_{x\to\infty}g^{even}(x)=\infty$. 
Let 
\[ M=\max\{|g^{even}(z)|:(g^{even})'(z)=0\},\] 
that is, $M$ is the maximum absolute value of the function $g^{even}_m(x)$ at the latter's critical points (which are the same as the critical points of $P^{even}_{m}(x)$, as the two functions differ by a constant). 
For any $m \geq \lfloor \tfrac{M}{R} \rfloor + 1$, the points on the graph of $P^{even}_m(x)$ whose horizontal values are critical points of $P^{even}_m(x)$ all lie above the horzontal axis. It follows that the roots of $P^{even}_m(x) = mR+g^{even}(x)$ are simple (that is, have multiplicity $1$), as if a root $r$ of $P^{even}_m(x)$ had multiplicity larger than $1$, then it would also be a critical point of $P^{even}_m(x)$, but for the chosen value of $m$, $P^{even}_m(r)>0$. Moreover, $P^{even}_m(x)$ has at most one real root, as if it had two roots $a < b$, then by the simpleness of the roots, either the function $P^{even}_m(x)$ is negative at some point between $a$ and $b$, or to the right of $b$, but in either case $P^{even}_m(x)$ would have a critical point $c$ at which $P^{even}_m(c) < 0$, a contradiction. 
In any event, as $P^{even}_m(x)$ has at most one real root (counting multiplicities) and $\deg(P^{even}_m(x)) \geq 2$, $P^{even}_{m}(x)$ must have a nonreal root. 
By the Hermite-Biehler theorem, it follows that $P_{m}(x)=m(x+R)+f(x+R)$ has a root in the RHP. Note that $x=a+ib$ is a root of $P_{m}(x)$ if and only if $x+R=(a+R)+ib$ is a root of $f(x)+mx$. Since there exists a root $x$ with $\text{Re}(x)\ge 0$ of $P_{m}(x)$, $x+R$ is a root of $f(x)+mx$ with $\text{Re}(x+R)\ge R$.  Therefore, for sufficiently large $m$, $f(x)+mx$ has roots with real part greater than $R$.

A similar (but slightly simpler) argument holds for part 2, provided $d \geq 4$, so all that remains is the case $d = 3$.
In this case, let $f(x)=a_0+a_1x+a_2x^2+a_3x^3$. Set 
\begin{eqnarray*}
g(x) & = & f(x+R)\\
 & =& a_0+a_1R+a_2R^2+a_3R^3+(a_1+2Ra_2+3R^2a_3)x+(3Ra_3+a_2)x^2+a_3x^3.
 \end{eqnarray*}
Now let $P_{\ell}=\ell+g(x)$. By Theorem \ref{thmhb}, $P_{\ell}$ is stable if and only if
$$-\frac{a_1+2Ra_2+3R^2a_3}{a_3} \leq -\frac{a_0+a_1R+a_2R^2+a_3R^3+\ell}{3Ra_3+a_2},$$
that is, if and only if 
$$\frac{a_1+2Ra_2+3R^2a_3}{a_3} \geq \frac{a_0+a_1R+a_2R^2+a_3R^3+\ell}{3Ra_3+a_2},$$
but clearly this fails if $\ell$ is large enough. 
Therefore, for $\ell$ sufficiently large, $P_{\ell}^{odd}\not\prec P_{\ell}^{even}$ and therefore, $f(x)$ has a root with real part greater than $R$.
\end{proof}

We shall shortly show the there are nonstable graphs of every independence number greater than $3$ by combining the previous lemma with another tool from complex analysis, the well known and useful {\em Gauss--Lucas} Theorem, which states that the convex hull of roots of polynomials only shrink when taking derivatives.

\begin{theorem}[\bfseries Gauss--Lucas]\label{thmgl}
Let $f(z)$ be a nonconstant polynomial with complex coefficients, and
let $f'(z)$ be the derivative of $f(z)$.  Then 
the roots of $f'(z)$ lie in the convex hull of the set of roots of $f(z)$. \hfill $\Box$
\end{theorem}

\begin{cor}\label{glcor}
If $f'(z)$ has a root $\gamma'$ with $\text{Re}(\gamma')=R$, then $f(z)$ has a root $\gamma$ such that $\text{Re}(\gamma)\ge R$. \hfill $\Box$
\end{cor}

We are now able to provide, for each $\alpha \geq 4$, infinitely many examples of graphs with independence number $\alpha$ that are nonstable. Moreover, we can embed {\em any} graph with independence number $\alpha \geq 4$ into another nonstable one with the same independence number, and we can even do so with a (nonreal) independence root as as far to the right as we like. To do this, we use the join operation. The {\em join} of two graphs $G$ and $H$, denoted $G+H$, is the graph obtained by joining all vertices of $G$ with all vertices of $H$.

\begin{prop}\label{prop:joingraphind3}
Let $G=H+\underbrace{F+F+\cdots+F}_\text{k}$, the join of a graph $H$ and $k$ copies of $F$. If $\alpha(H)\ge \alpha(F)+3$, then for $k$ sufficiently large, $i(G,x)$ has roots with arbitrarily large real part.  
\end{prop}
\begin{proof}
Let $R>0$. Assume $\alpha(H)\ge \alpha(F)+3$. Also let $c$ be the coefficient of $x^{\alpha(F)}$ in $i(F,x)$ (it is the number of independent sets of $F$ of maximum cardinality). We take the $\alpha(F)$-th derivative of $i(G,x)$, denoted $i^{\langle \alpha(F)\rangle}(G,x)$. Since $$i(G,x)=i(H,x)+ki(F,x)-(k-1),$$
we have
$$i^{\langle \alpha(F) \rangle}(G,x)=i^{\langle \alpha(F) \rangle}(H,x)+c\cdot k\cdot \alpha(F)!$$ 

Since $\alpha(H)\ge \alpha(F)+3$, the polynomial $i^{\langle \alpha(F) \rangle}(H,x)$ has degree at least $3$. We also know that $c$ and $\alpha(F) !$ are both at least $1$ so we may choose a sufficiently large $k$ and apply Lemma \ref{lemarblargeRHProots} to show that $i^{\langle \alpha(F) \rangle}(G,x)$ has (nonreal) roots with arbitrarily large real parts. By Corollary~\ref{glcor}, the same is true of $i(G,x)$.
\end{proof}

Since the independence number of a complete graph is $1$, the following corollary follows immediately.

\begin{cor}\label{propjoincliquelargerealpart}
Let $G$ be a graph with independence number at least $4$, and let $R>0$. Then for all $m$ sufficiently large, $i(G+K_m,x)$ has a root with real part greater than $R$.
\end{cor}

\begin{cor}\label{cor:indnum4subgraphunstable}
If $G$ is a graph with $\alpha(G)\ge 4$, then $G$ is an induced subgraph of a graph with independence number $\alpha(G)$ that is not stable.
\end{cor}
\begin{proof}
From Corollary \ref{propjoincliquelargerealpart}, $H=G+K_m$ is not stable for $m$ sufficiently large. Joining a clique does not change the independence number of the graph, so $\alpha(H)=\alpha(G)$ and $G$ is a subgraph of $H$.
\end{proof}
\section{Graphs with Stable Independence Polynomials}\label{secLHP}

While we have seen that graphs with small independence number are stable, what other families of graphs are stable? By direct calculations, graphs on up to at least $10$ vertices and trees on up to at least $20$ vertices have all their independence roots in the LHP. 
As noted earlier, a graph with all real independence roots is necessarily stable since the real independence roots must be negative (as independence polynomials have all positive coefficients). The Chudnovsky-Seymour result therefore implies that all claw-free graphs are stable. What about infinite stable families whose independence polynomials do \underline{not} have  all real roots? We begin by showing that stars (which include the claw $K_{1,3}$) are examples of such graphs. We make use of another well-known result from complex analysis, Rouch\'{e}'s Theorem (see, for example, \cite{FISHER}).

\begin{theorem}[\bfseries Rouch\'{e}'s Theorem]\label{thmrouche}
Let $f$ and $g$ be analytic functions on an open set containing $\gamma$, a simple piecewise smooth closed curve, and its interior. If $|f(z)+g(z)|<|f(z)|$ for all $z\in\gamma$, then $f$ and $g$ have the same number of zeros inside $\gamma$, counting multiplicities. \hfill $\Box$
\end{theorem}

\begin{figure}[ht]
\begin{center}
\includegraphics[scale=.3]{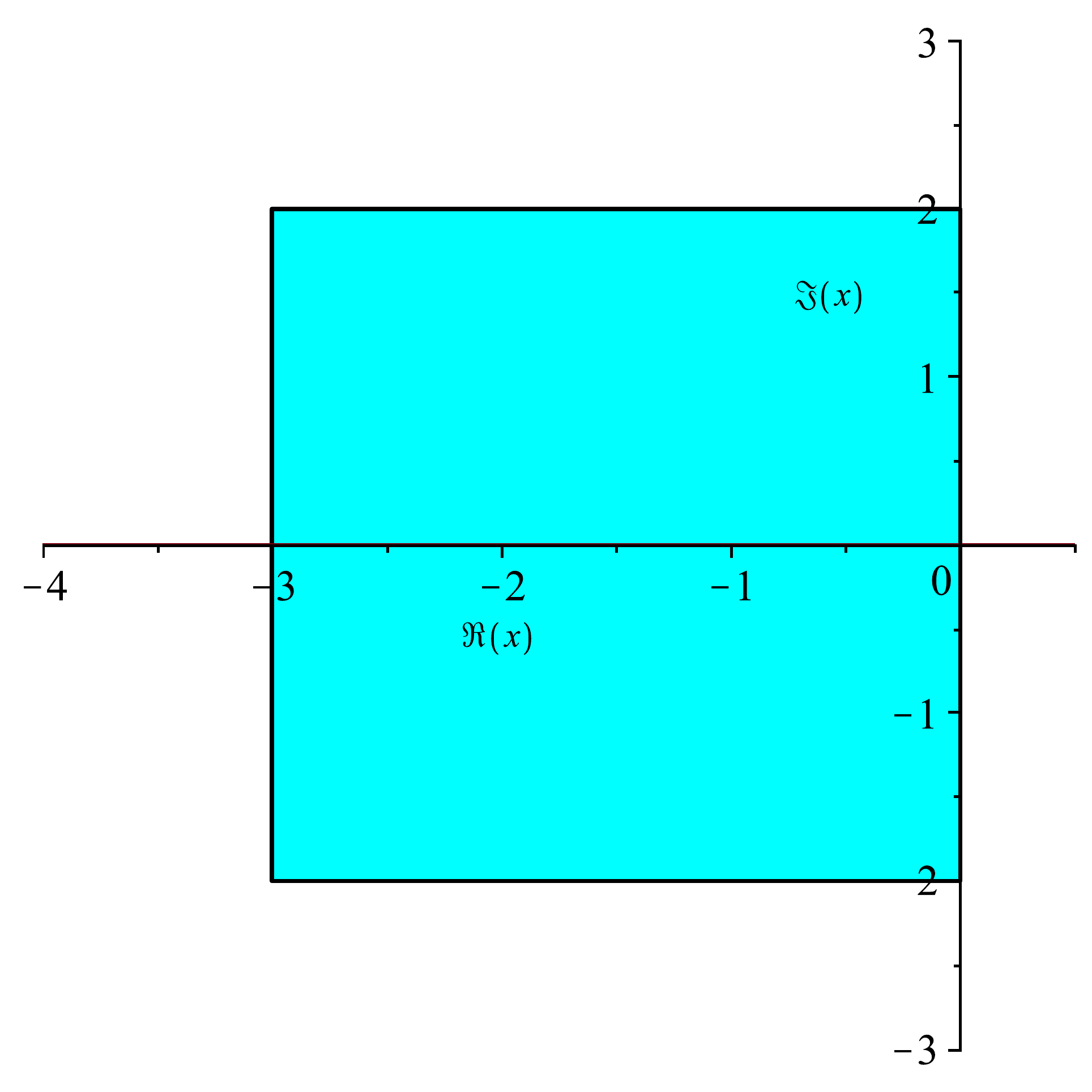}
\caption{The region $\gamma$ in Proposition \ref{propStarsstable}.}\label{figgamma}
\end{center}
\end{figure}

\begin{prop}\label{propStarsstable}
The roots of $i(K_{1,n},x)$ are in the left half-plane.
\end{prop}
\begin{proof}
Let $G=K_{1,n}$; then $i(G,x)=x+(1+x)^n$. Let $f(z)=-(1+z)^n$ and $g(z)=(1+z)^n+z$ and set 
\begin{itemize}
\item $\gamma_1=\{z:\text{Re}(z)=0\ \text{and}\ -2\le\text{Im}(z)\le2\}$, 
\item $\gamma_2=\{z:-3\le\text{Re}(z)\le0\ \text{and}\ \text{Im}(z)=2\}$, 
\item $\gamma_3=\{z:\text{Re}(z)=-3\ \text{and}\ -2\le\text{Im}(z)\le2\}$, and 
\item $\gamma_4=\{z:-3\le\text{Re}(z)\le0\ \text{and}\ \text{Im}(z)=-2\}$. 
\end{itemize}
Let $\gamma$ be the curve consisting of four line segments $\gamma_1,\gamma_2,\gamma_3,\gamma_4$, i.e. $\gamma=\gamma_1+\gamma_2+\gamma_3+\gamma_4$, see Figure \ref{figgamma}. The functions $f$ and $g$ are clearly analytic on $\Complex$ which contains $\gamma$ and its interior. The curve $\gamma$ is a simple piecewise smooth closed curve so the hypotheses of Rouch\'{e}'s Theorem are satisfied. 

We now show that $|f(z)+g(z)|<|f(z)|$ for all $z\in \gamma$ (we actually consider their squares to simplify computations). Note that $|f(z)+g(z)|=|z|$. As $i(K_{1,1},x)=1+2x$ has only one root at $\frac{-1}{2}$, we will assume $n\ge2$. 

\vspace{0.15in}
\noindent \textbf{Case 1:} If $z\in \gamma_1$, then $z=ki$ where $-2\le k\le2$. Now, $|z|^2=k^2$ and $|f(z)|^2=|1+x|^{2n}=(1+k^2)^{n}$. Clearly $k^2<(1+k^2)^{n}$, so it follows that $|f(z)+g(z)|^{2}<|f(z)|^{2}$, and hence $|f(z)+g(z)|<|f(z)|$ for all $z\in \gamma_1$.

\vspace{0.15in}
\noindent \textbf{Case 2:} If $z\in \gamma_2$, then $z=k+2i$ where $-3\le k\le0$.  In this case $|z|^2=k^2+4$ and $|(1+x)^{n}|^2=((1+k)^2+4)^{n}$. As in case 1, it suffices to show $((1+k)^2+4)^2>k^2+4$ since $n\ge 2$.  Now $h(k)=((1+k)^2+4)^2-k^2-4$ takes on the value $59$ at $k=1$ and it can be shown that $h(k)$ has no real roots. Therefore, $((1+k)^2+4)^2>k^2+4$ for all $k$ and hence $|f(z)+g(z)|<|f(z)|$ for all $z\in \gamma_2$.

\vspace{0.15in}
\noindent \textbf{Case 3:} If $z\in \gamma_3$, then $z=-3+ki$ where $-2\le k\le 2$. Now, $|z|^2=9+k^2$ and $|(1+z)^n|^2=(4+k^2)^{n}$. It suffices to show that $9+k^2<(4+k^2)^2$ since $n\ge 2$. Evaluating at $k=0$ $(4+k^2)^2-k^2-9$ takes on the value $7$ and it has no real roots. Hence the inequality holds for all $k$, and so $|f(z)+g(z)|<|f(z)|$ for all $z\in \gamma_3$. 

\vspace{0.15in}
\noindent \textbf{Case 4:} If $z \in \gamma_4$, then $z=k-2i$ where $-3\le k\le0$. If we set $w = \bar{z} = k+2i$, then $|z|^2 = |w|^2 = k^2 + 4$ and $|(1+z)^n| = |(1+w)^n| = ((1+k)^2+4)^2$ so we conclude our result from our proof for case 2.

\vspace{0.15in}
All cases together show that for all $z\in\gamma$, $|f(z)+g(z)|<|f(z)|$. Therefore by Rouch\'{e}'s Theorem we know that $f$ and $g$ have the same number of zeros inside $\gamma$ counting multiplicities. We know that $f$ has one root of multiplicity $n$ at $z=-1$ which is inside $\gamma$. Therefore $g(z)=i(G,z)$ has all $n$ of its roots in $\gamma$ which is contained in the (open) left half-plane.
\end{proof}

We now extend the star family to a much larger family of graphs that are also stable.
The \textit{corona} of a graph $G$ with a graph $H$, denoted $G\circ H$, is defined by starting with the graph $G$, and for each vertex $v$ of $G$, joining a new copy $H_{v}$ of $H$ to $v$. The graph $G\circ H$ has $|V(G)|+|V(G)||V(H)|$ vertices and $|E(G)|+|V(G)||E(H)|+|V(G)||V(H)|$ edges. For example, the star $K_{1,n}$ can be thought of as $K_1\circ \overline{K_n}$. See Figure \ref{figcorona} for an example of the corona of two other graphs. There is a nice relationship between the independence polynomials of $G$, $H$, and $G\circ H$ that was first described by Gutman \cite{Gutman1992}.

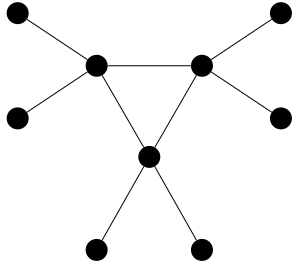
\begin{figure}[htp]
\def\c{0.7}
\def\r{1}
\centering
\scalebox{\c}{
\begin{tikzpicture}
\begin{scope}[every node/.style={circle,thick,draw,fill}]
    \node (1) at (-1*\r,0*\r) {};
    \node (2) at (1*\r,0*\r) {};
    \node (3) at (0*\r,-1.73205*\r) {};
    
    \node (4) at (2.5*\r,1*\r) {};
    \node (5) at (2.5*\r,-1*\r) {};
    
    \node (6) at (1*\r,-3.5*\r) {};
    \node(7) at (-1*\r,-3.5*\r) {};
    
    \node (8) at (-2.5*\r,1*\r) {};  
    \node (9) at (-2.5*\r,-1*\r) {};  
\end{scope}

\begin{scope}
    \path [-] (1) edge node {} (2);
    \path [-] (1) edge node {} (3);
    \path [-] (2) edge node {} (3);
    
    \path [-] (2) edge node {} (4);
    \path [-] (2) edge node {} (5);
    
    \path [-] (1) edge node {} (8);
    \path [-] (1) edge node {} (9);
    
    \path [-] (3) edge node {} (6);
    \path [-] (3) edge node {} (7);
\end{scope}

\end{tikzpicture}}
\caption{The graph $K_3\circ\overline{K_2}$}\label{figcorona}
\end{figure}

\begin{theorem}[\cite{Gutman1992}]\label{propcoronaformula}
If $G$ and $H$ are graphs with $G$ on $n$ vertices, then $$i(G\circ H,x)=i\left(G,\tfrac{x}{i(H,x)}\right)i(H,x)^n.$$
\end{theorem}

One special case of the corona product that is particularly useful is the corona with $K_1$. The end result is adding a pendant vertex to each vertex of the graph. The product $G\circ K_1$ is often denoted $G^{\ast}$ and called the \textit{graph star of $G$} \cite{ToppVolkmann1990, INDPOLY}; from above it has independence polynomial $$i(G\circ K_{1},x)=i\left(G,\tfrac{x}{i(H,x)}\right)(1+x)^n.$$
It is easily seen that $G^{\ast}$ is always very well-covered, that is, all maximal independent sets contain exactly half the vertex set.

For a graph $G$ and positive integer $k$, let $G^{k\ast}$ denote the {\it graph $k$--star of $G$}, that is, the graph formed by iteratively attaching pendant vertices $k$ times:
\[ G^{k\ast} = \left\{ \begin{array}{ll}
                               G^\ast & \mbox{if $k = 1$},\\
                               (G^{(k-1)\ast})^\ast & \mbox{if $k \geq 2$.}
                              \end{array}  
                     \right. 
\]   

We now show that the graph star operation preserves the stability of independence polynomials. The proof uses properties of {\em M\"{o}bius transformations}, which are rational functions of the form $$T(z)=\frac{az+b}{cz+d}$$ where $a,b,c,d,z\in \Complex$ and $ad-bc\neq 0$. More background on M\"{o}bius transformations can be found, for example, in section 3.3 of Fisher's book \cite{FISHER}.
\begin{prop}\label{thmrootsregion}
If the roots of $i(G,x)$ lie outside of the region bounded by the circle with with radius $\tfrac{1}{2}$ centred at $\tfrac{1}{2}$, then $i(G^*,x)$ is stable.
\end{prop}
\begin{proof}
Let $C$ be the circle with center $z = 1/2$ and radius $1/2$.
Note that the image of the imaginary axis,  $\{ z:\text{Re}(z)=0\}$, under the M{\"o}bius transformation $f(z)=\tfrac{z}{1+z}$ is $C$ (one need only observe that the image of the points $0$, $i$, and $-i$ are $0$, $\tfrac{1}{2}+\tfrac{1}{2}i$, and  $\tfrac{1}{2}-\tfrac{1}{2}i$, respectively). Moreover, as M{\"o}bius transformations send lines and circles to lines and circles, and the interiors/exteriors of circles and half-planes of lines to the same set, we find that the open right half-plane gets mapped to the interior of the circle $C$ (as $\tfrac{1}{2}$, which is in the open RHP, gets mapped to $\tfrac{1}{3}$, which is in the interior of $C$). It follows that the open LHP gets mapped to the exterior of $C$.

The roots of $i(G^*,x)$, along with $-1$ to some multiplicity, are found by solving $f(z)=r$ for every root $r$ of $i(G,x)$ since $i(G^*,x)=(1+x)^{n}i(G,\tfrac{x}{1+x})$ by Proposition \ref{propcoronaformula}. Therefore, if $i(G,x)$ has roots outside of $C$, then $i(G^*,x)$ is stable.

\end{proof}
\begin{figure}[ht]
\begin{center}
\includegraphics[scale=.3]{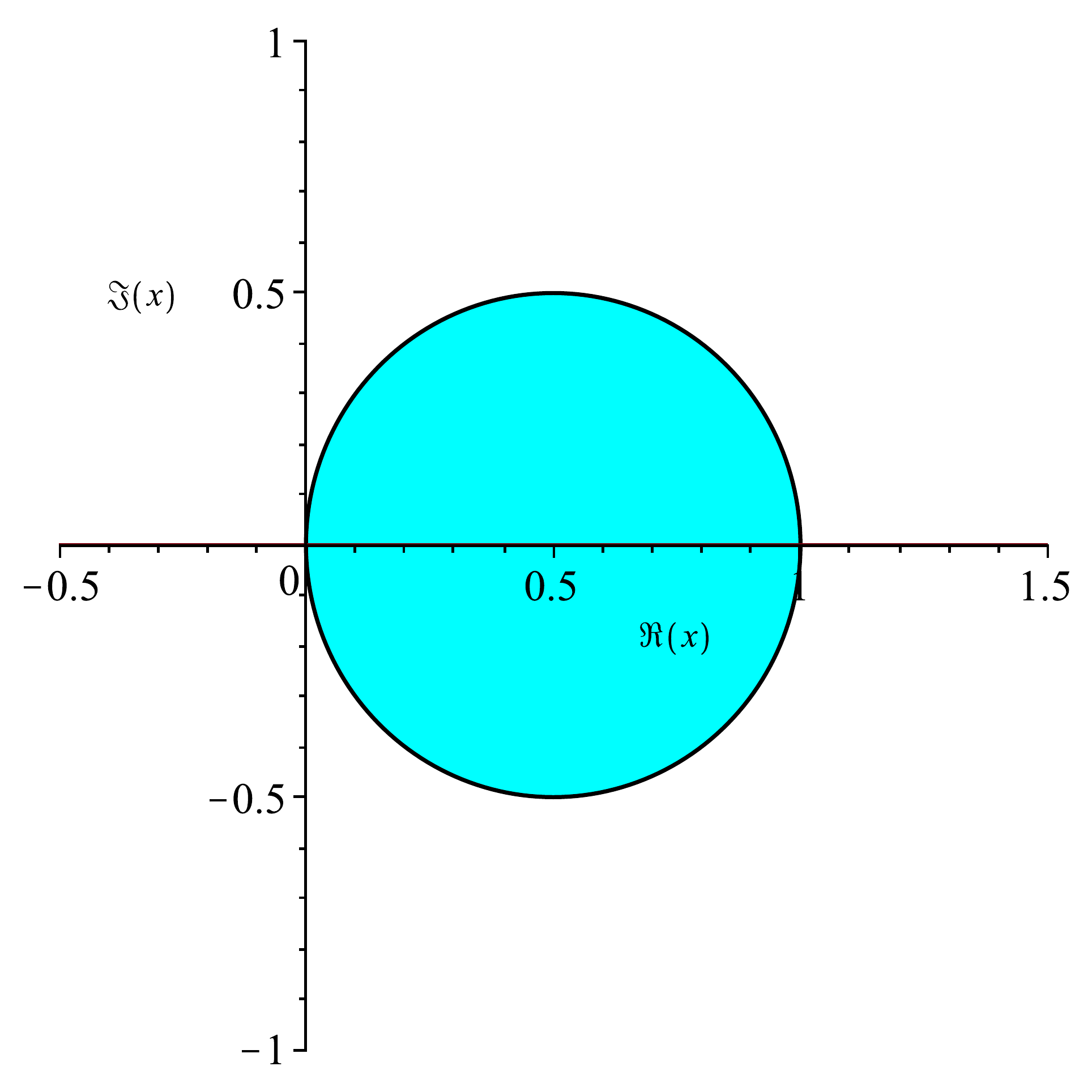}
\caption{Region in Proposition \ref{thmrootsregion}.}
\end{center}
\end{figure}

In Section \ref{secSmallIndNum} we will show that for $\alpha(G)\le 3$, $i(G,x)$ is always stable. That upcoming discussion together with Proposition \ref{thmrootsregion} and Theorem \ref{propStarsstable} proves the following corollary.

\begin{cor}\label{corLHPfamilies}
If $G$ is a claw-free graph, $G=K_{1,n}$, or $\alpha(G)\le 3$, then the graph $k$--star of $G$ is stable for all $k\ge 1$.
\end{cor}

Corollary \ref{corLHPfamilies} provides more families of stable graphs, but can the $k$--star be used to construct more families? It turns out it can be used to show that every graph is eventually stable after iterating the star operation enough times. To prove this we will need an extension of Theorem \ref{propcoronaformula} that works for $G^{k\ast}$ for any $\ge 1$.

\begin{prop}[\cite{browncameron2016}]\label{iteratedleafyprop} For any graph $G$ of order $n$ and any positive integer $k$,
$$i(G^{k\ast},x)=i(G,\tfrac{x}{kx+1})(kx+1)^n\prod_{\ell=1}^{k-1}\left(\ell x+1\right)^{n2^{k-\ell-1}}.$$
\end{prop}

We are now ready to prove our result.

\begin{theorem}
Let $G$ be a graph and $S$ be the set of its independence roots. If $$k>\max_{r\in S}\left\{\frac{\text{\emph{Re}}(r)}{|r|^2}\right\},$$ then $G^{k\ast}$ is stable.
\end{theorem}
\begin{proof}
Let $|V(G)|=n$ and $$k>\max_{r\in S}\left\{\frac{\text{Re}(r)}{|r|^2}\right\}.$$ Then by Proposition \ref{iteratedleafyprop}, 
$$i(G^{k\ast},x)=i(G,\tfrac{x}{kx+1})(kx+1)^n\prod_{\ell=1}^{k-1}\left(\ell x+1\right)^{n2^{k-\ell-1}}.$$
We know that the rational roots of the form $-\frac{1}{\ell}$ will surely all lie in the LHP so we must only consider the roots of $i(G,\tfrac{x}{kx+1})(kx+1)^{\alpha(G)}$ which can be found by solving for $z$ in $r=\tfrac{z}{kz+1}$ where $r\in S$, that is, $r$ is an independence root of $G$. Let $r\in S$, with $r=a+ib$ and consider the independence root of $G^{k\ast}$ of the form $z=\frac{r}{1-kr}$. Now we have that
\begin{eqnarray}
\text{Re}(z)=\frac{(a^2+b^2)(-k)+a}{(1-ka)^2+b^2}
\end{eqnarray}

So the sign of $\text{Re}(z)$ is the sign of $(a^2+b^2)(-k)+a=|z|^2(-k)$ and 
\begin{align*}
|z|^2(-k)+a&<|z|^2\left(-\frac{a}{|z|^2}\right)+a\\
&=-a+a\\
&=0.
\end{align*}
 
Therefore, $\text{Re}(z)<0$ for all independence roots of $G^{k\ast}$, $z$. Hence $G^{k\ast}$ is stable. 

\end{proof}

The next corollary provides an interesting contrast with different graph operations when compared with Corollary \ref{cor:indnum4subgraphunstable}.

\begin{cor}
Every graph is a subgraph of a stable graph.
\end{cor}

\section{Nonstable Families of Graphs}\label{secRHPgraphs}
We have seen that, starting with a graph with independence number at least $4$, joining a large clique produces nonstable graphs. In this section we provide more constructions that will produce families of nonstable graphs, the lexicographic product and the corona product. The last construction preserves acyclicity and therefore provides families of nonstable trees, which the construction of the previous section does not (and is surprising, given that we have noted that there are no roots in the RHP for trees of order at most 20).

%

The nonstable graph families that have been discovered so far have many vertices and we do not know the smallest nonstable graphs. We will provide some {\em relatively} small nonstable graphs via Sturm's sequences.

For a real polynomial $f$, the \textit{Sturm sequence} of $f$ is the sequence $f_0,f_1,\ldots,f_k$ where $f_0=f$, $f_1=f'$ and $f_i=-\text{rem}(f_{i-1},f_{i-2})$ for $i\ge 2$, where $\text{rem}(f_{i-1},f_{i-2})$ is the remainder when $f_{i-1}$ is divided by $f_{i-2}$ ($f_k$ is the last nonzero term in the sequence of polynomials of strictly decreasing degrees). Sturm sequences are a very useful tool for determining the nature of polynomial roots due to the following result (see \cite{Sturm}).

\begin{theorem}[\bfseries Sturm's Theorem]\label{thmSturm}
Let $f$ be a polynomial with real coefficients and $(f_0,f_1,\ldots,f_k)$ be its Sturm Sequence. Let $a<b$ be two real numbers that are not roots of $f$. Then the number of distinct roots of $f$ in $(a,b)$ is $V(a)-V(b)$ where $V(c)$ is the number of changes in sign in $(f_0(c),f_1(c),\ldots,f_k(c))$.
\end{theorem}

The Sturm sequence $(f_0,f_1,\ldots,f_k)$ of $f$ is said to have \textit{gaps in degree} if there is a $j\le k$ such that $\deg(f_j)<\deg(f_{i-1})-1$. If there is a $j\le k$ such that $f_j$ has a negative leading coefficient, the Sturm sequence is said to have a \textit{negative leading coefficient}. We now have the terminology to state the corollary of Sturm's Theorem (see \cite{BrownHickman2002}) that will be useful for our purposes.

\begin{cor}\label{corSturm}
Let $f$ be a real polynomial whose degree and leading coefficient are positive. Then $f$ has all real roots if and only if its Sturm sequence has no gaps in degree and no negative leading coefficients.
\end{cor}

We note that there are families of complete multipartite graphs that are stable. For example, stars are complete bipartite graphs and we have shown that they are stable. 
As well, it is not hard to see that  \[ i(K_{n,n,\ldots,n},x)=k(1+x)^n-(k-1).\] 
The roots of this are $z_k=\left(\frac{k-1}{k}\right)^{1/n}e^{2k\pi/n}-1$ for $k=0,1,\ldots,n-1$. Since $\left(\frac{k-1}{k}\right)^{1/n}<1$ for all $k\ge 1$, it follows that $\text{Re}(z_k)<0$ for all $k$. Therefore, $K_{n,n,\ldots,n}$ is stable for all $n$ and $k$. 

It may seem that all complete multipartite graphs are stable, but such is not the case. We will consider the graphs $K_{1,2,3,\ldots,n}$, the complete multipartite graph with one part of each of the sizes $1,2,\ldots,n$, and use Corollary \ref{corSturm} to prove that these graphs are not stable if $n \geq 15$. 

\begin{theorem}\label{thmmultipartite}
$i(K_{1,2,\ldots,n},x)$ is not stable for $n\ge 15$.
\end{theorem}
\begin{proof}
We will prove $n=15$ and $n=16$ directly, and then provide a more general argument for $n\ge 17$. For $n=15$,  $i(K_{1,2,\ldots,n},x)$ has a root with real part approximately $0.009053086185689$ and is therefore not stable. For $n=16$, $f_8$ of the Sturm sequence of the odd part of $i(K_{1,2,\ldots,n},x)$ is $$-{\frac{
1577448937796744128202619637524087852027658290220375925735260560}{
79627136162551065499783779429209235652424929298356031742670249}}.
$$ Thus, by Theorem~\ref{thmSturm} and Theorem~\ref{thmhb}, $K_{1,2,\ldots,n}$ is not stable.

Now let $n\ge 17$ and $G=K_{1,2,\ldots,n}$. The independence polynomial of $G$ can be written as 
\begin{align*}
i(G,x)&=\sum_{k=1}^n(1+x)^k-(n-1)\\
&=\frac{(1+x)^{n+1}-(1+x)}{x}-(n-1)\\
&=\frac{(1+x)^{n+1}-nx-1}{x}.
\end{align*}	
Since $i(G,0)=1$, it follows that the nonzero roots of $g=(1+x)^{n+1}-nx-1$ are precisely the roots of $i(G,x)$. Let $g^{odd}$ be the odd part of $g$ as in the Hermite-Biehler Theorem, so $g^{odd}=1+\binom{n+1}{3}x+\binom{n+1}{5}x^2+\cdots+\binom{n+1}{\ell}x^{(\ell-1)/2}$ where 
  \[
    \ell=\left\{
                \begin{array}{ll}
                  n & \mbox{if $n$ is odd}\\
                  n-1 & \mbox{if $n$ is even}\\
                \end{array}
              \right.
  \]
is the largest odd number for which $G$ has an independent with that size. Note that $g^{odd}(0)=1$ as well so $0$ is not a root of $g^{odd}$, therefore, the roots of $g_o$ are all real if and only if the roots of $f=x^ng^{odd}\left(\frac{1}{x}\right)=x^n+\binom{n+1}{3}x^{n-1}+\binom{n+1}{5}x^{n-2}+\cdots+\binom{n+1}{\ell}x^{n-(\ell-1)/2}$ are all real. We will find the Sturm sequence of $f$ and show that it has a negative leading coefficient to prove that $f$ has nonreal roots. 

Let the Sturm sequence of $f$ be $(f_0,f_1,\ldots,f_k)$ ( where $f_0=f$ and $f_1=f'$). Both $f_0$ and $f_1$ are nonzero and have positive leading coefficient. The leading coefficient of $f_2$ is calculated as 
\[ c_2=\frac{{n}^{5}}{36}-{\frac {2\,{n}^{4}}{45}}+\frac{{n}^{3}}{36}-\frac{{n}^{2}}{36}-\frac{n}{18}+{\frac{13}{180}},\] a polynomial in $n$. This polynomial has its largest real root at approximately $1.454179113$, so for $n\ge 2$, $c_2>0$. The third term in the sequence, $f_3$, has leading coefficient $$c_3={\tfrac { \left( n-2 \right)  \left( n-3 \right)  \left( 105\,{n}^{8}+5719\,{n}
^{7}-34103\,{n}^{6}+63299\,{n}^{5}-79478\,{n}^{4}+34046\,{n}^{3}+5068
\,{n}^{2}-15584\,n+55488 \right) 
 n}{35280\, \left( 5\,{n}^{3}-
8\,{n}^{2}+10\,n-13 \right) ^{2}}}.$$
The denominator of $c_3$ is defined and positive for all $n$ as it has no integer roots (easily verified by the Rational Roots Theorem). The numerator's largest real root is approximately $3.587037796$, and thus for $n\ge 4$, $c_3>0$.

We now consider the term $f_4$. The leading coefficient of this term is $$c_4=-{\tfrac { \gamma\left( n+1\right)\left( n-1 \right) \left( n-4 \right)  \left( n-5 \right)   
 \left( 5\,{n}^{3}-8\,{n}^{2}+10\,n-13 \right) ^{2} \left( n+2
 \right) }{40772160\, \left( 105\,{n}^{8}+5719\,{n}^{7}-34103\,{n}^{6}
+63299\,{n}^{5}-79478\,{n}^{4}+34046\,{n}^{3}+5068\,{n}^{2}-15584\,n+
55488 \right) ^{2}}}
$$
where 

\begin{align*}
\gamma=\ &1036035\,{n}^{14}-18710307\,{n}^{13}+60715080\,{n}^{
12}-1252685357\,{n}^{11}+16301479454\,{n}^{10}\\
&-71027287359,{n}^{9}+
150542755560\,{n}^{8}-194411482671\,{n}^{7}+73908295527\,{n}^{6}\\
&+81621340094\,{n}^{5}-183113161400\,{n}^{4}+127579216128\,{n}^{3}-
28712745216\,{n}^{2}\\
&+24221417472\,n+78617640960
\end{align*}
 
The denominator of $c_4$ has its largest root at approximately $3.587037796$, so for 
$n\ge 4$, the denominator is defined and positive. The largest root of the numerator is approximately $16.22715983$, therefore for $n\ge 17$, $c_4<0$. 

Since there are no gaps in degree, as we have ensured $c_2,c_3,$ and $c_4$ are nonzero, and the Sturm sequence has a negative leading coefficient, $c_4$, it follows by Corollary \ref{corSturm} that $f$, and therefore $g^{odd}$, has a nonreal root. Thus, by the Hermite-Biehler Theorem $g$, and therefore $i(G,x)$, is not stable. 

\end{proof}

The join has given us much to discuss in terms of nonstable graphs, but we turn now to another graph operation, the  lexicographic product, for constructing other nonstable graphs. 
The \textit{lexicographic product} (or graph substitution)  is defined as follows. Given graphs $G$ and $H$ such that $V(G)=\{ v_1,v_2,...,v_n\}$ and $V(H)=\{ u_1,u_2,...,u_k\}$, the lexicographic product of $G$ and $H$ which we will denote $G[H]$ is the graph such that $V(G[H] )=V(G)\times V(H)$ and $(v_i,u_l)\sim (v_j,u_m)$ if $v_i\sim _G v_j$ or $i=j$ and $u_l\sim _H u_m$.  The graph $G[H]$, can be thought of as substituting a copy of $H$ for each vertex of $G$.

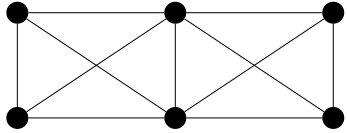
\begin{figure}[htp]
\def\c{0.7}
\def\r{1}
\centering
\scalebox{\c}{
\begin{tikzpicture}
\begin{scope}[every node/.style={circle,thick,draw,fill}]
    \node (1) at (-3*\r,0*\r) {};
    \node (2) at (0*\r,0*\r) {};
    \node (3) at (3*\r,0*\r) {};
    
    \node (4) at (-3*\r,-2*\r) {};
    \node (5) at (0*\r,-2*\r) {};
    \node (6) at (3*\r,-2*\r) {};
 
\end{scope}

\begin{scope}
    \path [-] (1) edge node {} (2);
    \path [-] (2) edge node {} (3);
    
    \path [-] (1) edge node {} (4);
    \path [-] (1) edge node {} (5);
    
    \path [-] (2) edge node {} (4);
    \path [-] (2) edge node {} (5);
    \path [-] (2) edge node {} (6);
    
    \path [-] (3) edge node {} (5);
    \path [-] (3) edge node {} (6);
    
    \path [-] (4) edge node {} (5);
    \path [-] (6) edge node {} (5);
  
\end{scope}

\end{tikzpicture}}
\caption{The lexicographic product $P_3[K_2]$.}
\end{figure}

The reason the lexicographic product has been so important to the study of the independence polynomial is due to the way the independence polynomials interact. 

\begin{theorem}[\cite{INDROOTS}]\label{thmlexicoprodforindpolys}
If $G$ and $H$ are graphs, then $i(G[H],x)=i(G,i(H,x)-1)$.
\end{theorem} 
 
In \cite{INDROOTS} it was shown that the independence roots of the family $\{P_n\}_{n\ge 1}$ are dense in $(-\infty,-\tfrac{1}{4}]$. This leads to another application of Lemma \ref{lemarblargeRHProots}. 

\begin{theorem}
If $H$ is a graph with $\alpha(H)\ge 4$, then for some $n$ sufficiently large, $P_n[H]$ has independence roots with arbitrarily large real part.
\end{theorem} 
\begin{proof}
Suppose $H$ is a graph with $\alpha(H)\ge 4$. By Theorem \ref{thmlexicoprodforindpolys}, we know that $i(P_n[H],x)=i(P_n,i(H,x)-1)$ and therefore the independence roots of $P_n[H]$ are found by solving $i(H,x)-1=r$, that is, $i(H,x)-r-1=0$ for all independence roots $r$ of $P_n$. Since we know that $\{P_n\}_{n\ge 1}$ are dense in $(-\infty,\tfrac{1}{4}]$, it follows that we can make $-r-1$ as large, in absolute value, as we like. Finally, since $\alpha(H)\ge 4$, Lemma \ref{lemarblargeRHProots} applies and $i(P_n[H],x)$ has roots with arbitrarily large real parts for $n$ sufficiently large.
\end{proof}

Finally, we consider the stability of trees (we recall that all trees of order $20$ and less have been found to be stable). Could this be true in general? As we have learned from the Chudnovsky-Seymour result on independence roots of claw-free graphs, a small restriction in the graph structure can have a large impact on the independence roots. Our previous constructions for producing graphs with roots arbitrarily far in the RHP did not turn up any trees, and it would be reasonable to speculate that perhaps all trees are stable, but this, in fact, turns out to be false.  Before we can provide a family of trees with nonstable independence polynomials, we first must show that there exist trees with real independence roots arbitrarily close to $0$.

\begin{lem}\label{lemmastarroots}
Fix $\varepsilon > 0$. Then for $n$ sufficiently large, there exists a real root $r$ of $i(K_{1,n},x)$ with $|r| < \varepsilon$.
\end{lem}
\begin{proof}
We know that $i(K_{1,n},x)=x+(1+x)^n$. Evaluating $i(K_{1,n},x)$ at $0$ we obtain $1$. Evaluating at $s=-\dfrac{1}{\ln(n)}$ we obtain 
\begin{align}
-\dfrac{1}{\ln(n)}+\left(1-\dfrac{1}{\ln(n)}\right)^n\label{negativevalue}
\end{align}

To show that (\ref{negativevalue}) is negative, we will require the following two elementary inequalities for $x\in (0,1)$: 
\begin{eqnarray*}
\ln(x) &\ge & \frac{x-1}{x} ~~~~~\mbox{and}\label{bound1}\\
\ln(1-x) &\le & -x\label{bound2} 
\end{eqnarray*}
Set $x=\frac{1}{\ln(n)}$; note that for $n\ge 3$, $x\in(0,1)$. 
If we can show that $x+x^2e^{\frac{1}{x}}-1>0$, then the following sequence of implications hold:

\begin{eqnarray*}
x+x^2e^{\frac{1}{x}}-1&>& 0\\
\frac{x-1}{x}&>&-xe^{\frac{1}{x}}\\ 
\ln(x)&>&\ln(1-x)e^{\frac{1}{x}}\ \ \ \ \ \text{by inequalities (\ref{bound1}) and (\ref{bound2})}\\
\ln\left(\frac{1}{\ln(n)}\right)&>&n\ln\left(1-\frac{1}{\ln(n)}\right)\\
-\dfrac{1}{\ln(n)}+\left(1-\dfrac{1}{\ln(n)}\right)^n & < & 0.
\end{eqnarray*}

We now show that $h(x) = x+x^2e^{\frac{1}{x}}-1>0$ is indeed true for $x\in(0,1)$. Now $h'(x)=1+e^{\frac{1}{x}}(2x-1)$ and $h''(x)=e^{\frac{1}{x}}\left(\frac{2x^2-2x+1}{x^2}\right)$. It is straightforward to see that for $x\neq0$, $h''(x)>0$ and therefore $h(x)$ is always concave up. The function $h'(x)$ is continuous on $[a,b]$ for all $0<a< b$ so, by Rolle's Theorem, $h(x)$ has at most one positive critical point. The Intermediate Value Theorem gives that any positive critical point of $h(x)$ must lie in the interval $(0.4,0.5)$ since $h'(0.4)\approx-1.436498792<0$ and $h'(.5)=1>0$. Now on $(0.4,0.5)$, $h(x)\ge 0.4+(0.4)^2e^{2}-1\approx .582248976>0$. Thus $h(x)$ is concave up for all $x\neq 0$, has only one positive critical point which is in the interval $(0.4,0.5)$, and is strictly positive on $(0.4,0.5)$. It follows that the absolute minimum of $h(x)$ on $(0,\infty)$ is strictly positive, and therefore $h(x) = x+x^2e^{\frac{1}{x}}-1>0$ for all $x\in(0,1)$. We conclude that $i(G,s)<0$.

Let $n> e^{1/\varepsilon}$. We may conclude, by the Intermediate Value Theorem, that $i(K_{1,n},x)$ has a real root $r$ in the interval $\left(-\frac{1}{\ln(n)},0\right)$. Then
\begin{eqnarray*}
r&>&-\dfrac{1}{\ln(n)}\\
&>&-\dfrac{1}{\ln(e^{1/\varepsilon})}\\
&=&\dfrac{-1}{\varepsilon^{-1}}\\
&=&-\varepsilon.
\end{eqnarray*}

Hence, $i(G,x)$ has a root in $(-\varepsilon,0)$ for all $\varepsilon>0$.

\end{proof}

We can now prove that trees do not necessarily have stable independence polynomials (and in fact can have independence roots with arbitrarily large real part).  

\begin{prop}\label{proptreescoronalargerealpart}
If $G$ is a graph with $\alpha(G)\ge 4$ and $R>0$, then for sufficiently large $n$, $K_{1,n}\circ G$ has independence roots in the RHP with real part at least $R$.
\end{prop}
\begin{proof}
Set  $H=K_{1,n}\circ G$. By Theorem \ref{propcoronaformula}, $$i(H,x)=\left(i(G,x)\right)^{n+1}i\left(K_{1,n},\frac{x}{i(G,x)}\right)$$ so the independence roots of $H$ are the roots of $i(G,x)$ together with the roots of the polynomials $f(x)=-\frac{x}{r}+i(G,x)$ for all independence roots $r$ of $i(K_{1,n},x)$. By Lemma \ref{lemmastarroots}, there exist real roots $r$ that are negative and arbitrarily close to $0$ for sufficiently large $n$. In this case, $-\frac{x}{r}=px$ for some $p>0$ and since $\alpha(G)\ge 4$, we can apply Lemma \ref{lemarblargeRHProots} to show that for any $R> 0$ and sufficiently large $n$, $f(x)$ has a root with real part greater than $R$, and so the same holds for $i(H,x)$.
\end{proof}

This proposition implies that trees are not necessarily stable as $K_{1,n}\circ\overline{K_m}$ is a tree for all $m\ge 1$, and will have independence roots in the RHP for all $m\ge 4$ and sufficiently large $n$. Thus the independence roots of trees can be found with arbitrarily large real parts.

\section{Concluding Remarks}

We end this paper with a few open problems. First, we have seen that stars are stable, while other complete multipartite graphs are not. Calculations suggest that complete bipartite graphs are stable, so we ask:

\begin{problem}
Are all complete bipartite graphs stable?
\end{problem}

While we have provided a number of families (and constructions) of nonstable graphs, we still feel that stableness is a more common property, as small graphs suggest. For fixed $p \in (0,1)$, a \textit{random graph}, $G_{n,p}$, is a graph constructed on $n$ vertices where each pair of vertices is joined by an edge independently with probability $p$. Almost all graphs are said to have a certain property if as $n$ tends to $\infty$, the probability that $G_{n,p}$ has that property tends to $1$.

\begin{problem}
Are almost all graphs stable?
\end{problem}

While we have shown that every graph with independence number at most $3$ is stable, and there are nonstable graphs of all higher independence numbers, we ask:

\begin{problem}
Characterize when a graph of independence number $4$ is stable.
\end{problem}

\begin{problem}
Characterize when a tree is stable.
\end{problem}

Finally, we do not know the smallest graph with respect to vertex set size that is not stable. Calculations show that the cardinality is greater than $10$ and our arguments show that it is at most $25$, but it would be interesting to locate the extremal graph.


\end{document}